\documentclass[a4paper,oneside, 11 pt, reqno]{amsart}
\usepackage{enumerate}
\usepackage{amsmath,amsthm,amscd,amsfonts,amssymb,mathrsfs}
\usepackage{latexsym,amsmath}
\usepackage{mathrsfs}
\makeatletter
\@namedef{subjclassname@2010}{%
\textup{2010} Mathematics Subject Classification}
\makeatother

\newcommand{\ncom}{\newcommand}
\ncom{\beqn}{\begin{eqnarray*}}
\ncom{\eeqn}{\end{eqnarray*}}
\ncom{\beq}{\begin{eqnarray}}
\ncom{\eeq}{\end{eqnarray}}
\ncom{\cal}{\mathcal}
\ncom{\eop}{\hfill{{\rule{2.5mm}{2.5mm}}}}
\ncom{\eoe}{\hfill{{\rule{1.5mm}{1.5mm}}}}
\ncom{\eof}{\hfill{{\rule{1.5mm}{1.5mm}}}}
\ncom{\hone}{\mbox{\hspace{1em}}}
\ncom{\htwo}{\mbox{\hspace{2em}}}
\ncom{\hthree}{\mbox{\hspace{3em}}}
\ncom{\hfour}{\mbox{\hspace{4em}}}
\ncom{\hsev}{\mbox{\hspace{7em}}}
\ncom{\vone}{\vskip 2ex}
\ncom{\vtwo}{\vskip 4ex}
\ncom{\vonee}{\vskip 1.5ex}
\ncom{\vthree}{\vskip 6ex}
\ncom{\vfour}{\vspace*{8ex}}
\ncom{\norm}{\|\;\;\|}
\ncom{\integ}[4]{\int_{#1}^{#2}\,{#3}\,d{#4}}
\ncom{\inp}[2]{\langle{#1},\,{#2} \rangle}
\ncom{\Inp}[2]{\Langle{#1},\,{#2} \Langle}
\ncom{\vspan}[1]{{{\rm\,span}\#1 \}}}
\ncom{\dm}[1]{\displaystyle {#1}}

\newtheorem{theorem}{\bf Theorem}[section]
\newtheorem{proposition}[theorem]{\bf Proposition}%[section]
\newtheorem{lemma}[theorem]{\bf Lemma}%[section]

\newtheoremstyle
    {remarkstyle}
    {}
    {11pt}
    {}
    {}
    {\bfseries}
    {:}
    {     }
    {\thmname{#1} \thmnumber{#2} }

\theoremstyle{remarkstyle}

\newtheorem{remark}[theorem]{\bf Remark}%[section]
\newtheorem{definition}[theorem]{\bf Definition}%[section]
\newtheorem{example}[theorem]{\bf Example}

\begin{document}

\title[An Analytic Model for left invertible WTS]{An Analytic Model for left invertible Weighted Translation Semigroups}

\author[G. M. Phatak]{Geetanjali M. Phatak}
\address{Department of Mathematics, S. P. College\\
Pune- 411030, India}
\email{gmphatak19@gmail.com}

\author[V. M. Sholapurkar]{V. M. Sholapurkar}
\address{ Department of Mathematics, S. P. College\\
Pune- 411030, India}
\email{vmshola@gmail.com}

\date{}

\begin{abstract}
M. Embry and A. Lambert initiated the study of a semigroup of operators $\{S_t\}$ indexed by a non-negative real number $t$ and termed it as weighted translation semigroup. The operators $S_t$ are defined on $L^2(\mathbb R_+)$ by using a  weight function. The operator $S_t$ can be thought of as a continuous analogue of a weighted shift operator.
In this paper, we show that every left invertible operator $S_t$ can be modeled as a multiplication by $z$ on a reproducing kernel Hilbert space $\cal H$ of vector-valued analytic functions on a certain disc centered at the origin and the reproducing kernel associated with $\cal H$ is a diagonal operator. 
As it turns out that every hyperexpansive weighted translation semigroup is left invertile, the model applies to these semigroups. We also describe the spectral picture for the left invertible weighted translation semigroup. In the process, we point out the similarities and differences between a weighted shift operator and an operator $S_t.$ 
\end{abstract}
\maketitle
%\maketitle
%\tableofcontents
%\setcounter{tocdepth}{1}
\section{Introduction}
With a view to develop a continuous analogue of weighted shift operators, M. Embry and A. Lambert initiated the study of operators that can be defined by using a {\it weight function} (rather than a weight sequence)(refer to \cite{EL1}). In fact, for a positive measurable function $\varphi,$ they constructed a semigroup $\{S_t\}$ of bounded linear operators on $L^2(\mathbb R_+),$ parametrized by a non-negative real number $t$ and termed it as {\it weighted translation semigroup}. In \cite{PS}, we continued the work carried out in \cite{EL1},\cite{EL2} and obtained characterizations of hyperexpansive semigroups in terms of their symbols. 
In this paper, we further explore the semigroup $\{S_t\}$ and record some important properties. In our approach, the emphasis is on the study of an operator $S_t,$ for a fixed parameter $t$ and not on the semigroup structure of  $\{S_t\}.$ In the process, we highlight the similarities and differences between a weighted shift operator and an operator $S_t.$ 
In section 2, we set the notation and record some definitions required in the sequel. 
As described in \cite{PS} the term {\it hyperexpansive weighted translation semigroups} include completely hyperexpansive, 2-hyperexpansive, 2-isometric and alternatingly hyperexpansive weighted translation semigroups. 
In section 3, following the process given by S. Shimorin in \cite{Sh}, we construct an analytic model for a left invertible semigroup $\{S_t\}$. We say that the {\it semigroup $\{S_t\}$ is left invertible} if every operator $S_t$ in that semigroup is left invertible. 
Recall that a bounded linear operator $T$ on a Hilbert space $H$ is said to be  {\it analytic} if $\displaystyle \cap_{n\geq 0}\,T^nH=\{0\}.$ The multiplication operator ${\cal M}_z$ on $z$-invariant reproducing kernel Hilbert space of analytic functions defined on a disc in $\mathbb C$ is an example of an analytic operator. A result of S. Shimorin \cite{Sh} asserts that any left invertible analytic operator is unitarily equivalent to the multiplication operator ${\cal M}_z$ on a reproducing kernel Hilbert space $\cal H$ of vector-valued analytic functions defined on a certain disc. In this paper, we prove that the semigroup $\{S_t\}$ is analytic in the sense that the operator $S_t$ is analytic for each $t>0$ and obtain the corresponding reproducing kernel Hilbert space, so as to realize each such left invertible $S_t$ as a multiplication by $z.$ We observe that the reproducing kernel associated with $\cal H$ is diagonal. Also, $\cal H$ admits an orthonormal basis consisting of polynomials in $z.$ We explore the reproducing kernels associated with $\cal H$ in some special types of semigruops $\{S_t\}.$ In this section, we also prove that the semigroup $\{S_t\}$ has wandering subspace property. As a consequence, we prove that a left invertible operator $S_t,~t>0$ is an operator valued weighted shift. We recall that every hyperexpansive weighted translation semigroup is left invertible (\cite[Remark 4.2]{PS}). Thus the analytic model developed in this section applies to hyperexpansive weighted translation semigroups, allowing us to view these operators as the multiplication by $z$ on a suitable reproducing kernel Hilbert space.  
In section 4, we describe the spectral picture of a left invertible weighted translation semigroup $\{S_t\}$. We know that for a weighted shift operator, the point spectrum is empty and the spectrum is a always a disc. We observe that these properties are also shared by a left invertible operator $S_t,t>0.$ We also obtain a formula for the spectral radius of $S_t$ in terms of the symbol $\varphi.$ However, we point out that unlike a weighted shift operator, for an operator $S_t, t>0,$  the kernel of ${S_t}^*$ is infinite dimensional. 

\section{Preliminaries}
Let $\mathbb R_+$ be the set of non-negative real numbers and let $L^2({\mathbb R_+})$ denote the Hilbert space of complex valued square integrable Lebesgue measurable functions on $\mathbb R_+.$ Let ${\cal B}(L^2)$ denote the algebra of bounded linear operators on $L^2({\mathbb R_+}).$ 
\begin{definition}
For a measurable, positive function $\varphi$ defined on $\mathbb R_+$ and $t\in \mathbb R_+,$ define the function $\varphi_t : \mathbb R_+ \rightarrow \mathbb R_+$ by 
\begin{equation*}
\varphi_t(x) =
\begin{cases}
\displaystyle \sqrt {\frac{\varphi(x)}{\varphi(x-t)}} & \text{if~ $x\geq t$},\\
0 & \text{if~ $x<t$}.
\end{cases}
\end{equation*}
Suppose that $\varphi_t$ is essentially bounded for every $t \in \mathbb R_+$.
For each fixed $t\in \mathbb R_+,$ we define $S_t$ on $L^2({\mathbb R_+})$ by  
\begin{equation*}
S_tf(x) =
\begin{cases}
\varphi_t(x)f(x-t) & \text{if~ $x\geq t$},\\
0 & \text{if~ $x<t$}.
\end{cases}
\end{equation*}
\end{definition}
\begin{remark}
Substituting $\varphi_t$ in the above definition, we get
\begin{equation*}
S_tf(x) =
\begin{cases}
\displaystyle \sqrt {\frac{\varphi(x)}{\varphi(x-t)}}f(x-t) & \text{if~ $x\geq t$},\\
0 & \text{if~ $x<t$}.
\end{cases}
\end{equation*}
It is easy to see that for every $t\in \mathbb R_+,~S_t$ is a bounded linear operator on $L^2({\mathbb R_+})$ with $\|S_t\|=\|\varphi_t\|_\infty,$  where $\|\varphi_t\|_\infty$ stands for the essential supremum of $\varphi_t$ given by \beqn \|\varphi_t\|_\infty= \inf  \{M \in \mathbb R: \varphi_t(x)\leq M ~\rm{almost~everywhere}\}. \eeqn
The family $\{S_t:t\in \mathbb R_+\}$ in ${\cal B}(L^2)$ is a semigroup with $S_0=I,$ the identity operator and for all $t,s~\in \mathbb R_+$, $S_t\circ S_s=S_{t+s}.$ 
\end{remark}

We say that $\varphi_t$ is a {\it weight function corresponding to the operator $S_t$}. Further, the semigroup $\{S_t:t\in \mathbb R_+\}$ is referred to as the {\it weighted translation semigroup with symbol $\varphi$}.
Throughout this article, we assume that the symbol $\varphi$ is a continuous function on $\mathbb R_+.$ 

In \cite[Corollary 3.3]{PS}, characterizations of some special types of the semigroup $\{S_t\}$ such as  subnormal contractions, completely hyperexpansive, 2-hyperexpansive, alternatingly hyperexpansive and $m$-isometries are obtained in terms of their symbols.  
The special classes of functions characterizing these classes of operators have been studied extensively in the literature (refer to \cite{BCR},\cite{SSV},\cite{W-1}). 
%We record the definitions of these functions for the sake of completeness. 
%A $C^\infty$ function $f:\mathbb R_+ \rightarrow \mathbb R$ is called  
%\begin{enumerate} 
%\item 
%{\it completely monotone} if $$(-1)^kf^{(k)}(x)\geq 0, ~{\rm for~ all}~ k\geq 0,$$ where $f^{(k)}$ denotes the k$^{\rm {th}}$ derivative of $f.$
%\item 
%{\it completely alternating} if $$(-1)^{k-1} f^{(k)}(x)\geq 0,~~\mbox{for all}~ k\geq 1.$$
%\item 
%{\it absolutely monotone} if $$f^{(k)}(x)\geq 0, ~{\rm for~ all}~ k\geq 0.$$ 
%\end{enumerate}
We find it convenient to record the definitions of the classes of operators under consideration for ready reference. For a detailed account on these classes of operators, the reader is referred to \cite{AS},\cite{At},\cite{Co},\cite{Jb1},\cite{Jb2},\cite{SA}. 
Let $T$ be a bounded linear operator on a Hilbert space $H$ and $n$ be a positive integer. 
Let $B_n(T)$ denote the operator
\begin{equation} \label{B-n-T}
B_n(T)= \sum_{k=0}^n(-1)^k{n\choose k} {T^*}^kT^k.
\end{equation}
An operator $T$ is said to be 
\begin{enumerate} 
\item  {\it subnormal} if there exist a Hilbert space $K$ containing $H$ and a normal operator $N\in\cal B(K)$ such that $N{H}\subseteq H$ and $N|_{H}=T.$
\item {\it completely hyperexpansive} if 
$B_n(T) \leq 0,~ \rm{for~all~integers~} n\geq 1.$
\item {\it $m$-hyperexpansion} if 
$B_n(T) \leq 0,~ \rm{for~all~integers~} n,~ 1\leq n\leq m.$
\item {\it alternatingly hyperexpansive} if $$\sum_{k=0}^n(-1)^{n-k}{n\choose k} {T^*}^kT^k\geq 0, ~\rm{for~all~integers~} n\geq 1.$$  
\item {\it $m$-isometry} if $B_m(T)= 0.$ 
\item {\it hyponormal} if $T^*T-TT^*\geq 0.$ 
\item {\it contraction (expansion)} if $I-T^*T\geq 0 ~~(I-T^*T\leq 0).$
\end{enumerate} 

\section{Analytic model for weighted translation semigroup $\{S_t\}$}
Recall that a bounded linear operator $T$ on a Hilbert space $H$ is analytic if $\displaystyle \cap_{n\geq 0}T^nH=\{0\}.$ We say that the semigroup $\{S_t\}$ is analytic, if for every $t>0,$ the operator $S_t$ is analytic. 
In this section, we prove that the semigroup $\{S_t\}$ is analytic.  We further show that every left invertible operator $S_t,t>0$ can be modeled as a multiplication operator $ {\cal M}_z$ on a reproducing kernel Hilbert space (RKHS) $\cal H$ of vector-valued analytic functions on a disc centered at the origin. It turns out that the reproducing kernel associated with $\cal H$ is a diagonal operator. Also $\cal H$ admits an orthonormal basis consisting of polynomials in $z.$ We compute the reproducing kernels in some concrete cases. 

\subsection{Analytic property of $\{S_t\}$}
The following result is crucial in the construction of an analytic model for an operator $S_t.$  
\begin{theorem} \label{p17} For every $t>0,$ the operator $S_t$ is analytic. \end{theorem}
\begin{proof} We first observe that $S_t^k(L^2(\mathbb R_+))\subseteq \chi_{[kt,\infty )}L^2(\mathbb R_+)$ for $k=0,1,2,\cdots ,$\\$t>0,$ where $\chi_A$ stands for the characteristic function of the set $A.$  
For $f\in L^2(\mathbb R_+),$  
\begin{equation*}
(S_t^kf)(x) =
\begin{cases}
\displaystyle \sqrt{\frac{\varphi(x)}{\varphi(x-kt)}}f(x-kt) & \text{if~ $x\geq kt$},\\
0 & \text{if~ $x<kt$}.
\end{cases}
\end{equation*}
Therefore $S_t^kf \in \chi_{[kt,\infty )}L^2(\mathbb R_+).$ 
We now prove that for every $t>0,$ the operator $S_t$ is analytic. In view of the observation above, it is sufficient to prove that 
$\cap_{k=0}^\infty \chi_{[kt,\infty )}L^2(\mathbb R_+)=\{0\}.$ 
Let $f\in \displaystyle \cap _{k=0}^\infty \chi_{[kt,\infty )}L^2(\mathbb R_+).$ Then $f(x)=0$ for all $x<kt,$ for each $k\geq 0.$ This forces that $f(x)=0$ for all $x\in [0,\infty).$ Thus we have 
$\displaystyle \cap _{k=0}^\infty \chi_{[kt,\infty )}L^2(\mathbb R_+)=\{0\}$ implying that for every $t>0,$ the operator $S_t$ is analytic.   \end{proof} 

We now show that the semigroup $\{S_t\}$ possesses wandering subspace property in the sense that for every $t>0,$ the operator $S_t$ has wandering subspace property. Recall that an operator $T$ on a Hilbert space $H$ possesses the {\it wandering subspace property} if $H=[E]_T,$ where $[E]_T$ denotes the smallest $T$-invariant subspace containing $E=$ ker $T^*.$ 
The following lemma gives a useful description of kernel of the adjoint $S_t^*,$ for $t>0.$
\begin{lemma} \label{p27} For $t>0,$ ker $S_t^*=E=\chi_{[0,t)}L^2(\mathbb R_+).$ In particular, ker $S_t^*$ is infinite dimensional. \end{lemma}
\begin{proof} Let $g \in \chi_{[0,t)}L^2(\mathbb R_+).$ 
Then $g(x)=\chi_{[0,t)}(x)f(x)$ almost everywhere on $\mathbb R_+$ for some function $f\in L^2(\mathbb R_+).$ Now
$$(S_t^*g)(x) = (S_t^*\chi_{[0,t)}f)(x) =\sqrt{\frac{\varphi(x+t)}{\varphi(x)}}\chi_{[0,t)}(x+t)f(x+t)=0,~\rm{for ~all} ~x\geq 0.$$  Therefore $g\in E.$ 
Conversely, assume that $f\in E.$ Then $$(S_t^*f)(x)=\sqrt{\frac{\varphi(x+t)}{\varphi(x)}}f(x+t)=0,~\rm{for ~all} ~x\geq 0.$$ This implies $f(x+t)=0$ almost everywhere on $\mathbb R_+.$ Hence $f$ vanishes almost everywhere on the interval $[t,\infty).$ Therefore $f\in \chi_{[0,t)}L^2(\mathbb R_+).$ \end{proof} 

\begin{remark} In the light of the fact that for a weighted shift operator $T,$ ker $T^*$ is one dimensional \cite[Proposition 6.3]{Co}, Lemma \ref{p27} indicates that the theories of weighted shift operators and weighted translation semigroups are intrinsically different. \end{remark}

\begin{proposition} \label{p19} For every $t>0,$ the operator $S_t$ possesses wandering subspace property.\end{proposition}
\begin{proof} 
It is sufficient to prove that $\displaystyle (\vee _{n=0}^\infty S_t^nE)^\bot =\{0\}.$
Let $f\in \displaystyle (\vee _{n=0}^\infty S_t^nE)^\bot .$ Thus $f$ is orthogonal to 
$(S_t^nE) ~\rm{for~ all}~ n\geq 0.$
For a non-negative integer $n,$ we now prove that $f(x)=0$ almost everywhere on the interval $[nt,(n+1)t).$
Let $g(x)=f(x+nt)$ almost everywhere on $[nt,(n+1)t).$ By Lemma \ref{p27}, $\chi_{[0,t)}g \in E ~{\rm{and}}~ S_t^n\chi_{[0,t)}g \in S_t^nE.$ Now 
\begin{equation*}
(S_t^n\chi_{[0,t)}g)(x) =
\begin{cases}
\displaystyle \sqrt{\frac{\varphi(x)}{\varphi(x-nt)}}\chi_{[0,t)}(x-nt)g(x-nt) & \text{if~ $x\geq nt$},\\
0 & \text{if~ $x<nt$}.
\end{cases}
\end{equation*}
So that 
\begin{equation*}
(S_t^n\chi_{[0,t)}g)(x)=
\begin{cases}
\displaystyle \sqrt{\frac{\varphi(x)}{\varphi(x-nt)}}g(x-nt) & \text{if~ $x\in [nt,(n+1)t)$},\\
0 & \text{otherwise}.
\end{cases}
\end{equation*}
Now $\left\langle f, S_t^n\chi_{[0,t)}g \right\rangle =0$ implies that 
$$\int_{nt}^{(n+1)t} \sqrt{\frac{\varphi(x)}{\varphi(x-nt)}}f(x)\overline{g(x-nt)}dx=0.$$ 
As $\overline{g(x-nt)}=\overline{f(x)},$ we have
$$\int_{nt}^{(n+1)t} \sqrt{\frac{\varphi(x)}{\varphi(x-nt)}}|f(x)|^2dx=0.$$
This implies that $|f(x)|^2=0$ almost everywhere on the interval $[nt,(n+1)t).$ This result is true for every non-negative integer $n.$ Hence $f(x)=0$ almost everywhere on the interval $[0,\infty ).$
Hence for every $t>0,$ the operator $S_t$ possesses the wandering subspace property. \end{proof}

\subsection{Construction of an analytic model}
Let $\{S_t\}$ be a left invertible weighted translation semigroup. 
We now proceed to construct an analytic model for every operator in such a semigroup. 
A notion of the {\it  Cauchy dual} of a left invertible operator was introduced by S. Shimorin in \cite{Sh}. Recall that for a left invertible operator $T,$ the {\it Cauchy dual} $T^\prime$ of $T$ is defined as $T^\prime =T(T^*T)^{-1}.$
Note that a condition that for every $t\in \mathbb R_+, \inf_x \frac{\varphi(x+t)}{\varphi(x)} > 0,$ ensures the left invertibility of the semigroup $\{S_t\}.$ 
Now it is easy to see that the Cauchy dual  $S_t^{\prime}$ of $S_t$ is given by  
\begin{equation*}
S_t^{\prime}f(x) =
\begin{cases}
\displaystyle \frac{1}{\varphi_t(x)}f(x-t) & \text{if~ $x\geq t$},\\
0 & \text{if~ $x<t$}.
\end{cases}
\end{equation*}
Observe that for $t\in \mathbb R_+,$ the family of operators $\{S_t^{\prime}\}$ also forms a semigroup. 
We say that the weighted translation semigroup $\{S_t^{\prime}\}$ is a {\it Cauchy dual of the weighted translation semigroup $\{S_t\}.$ }
Recall that a result of S. Shimorin \cite{Sh} asserts that any left invertible analytic operator $T$ is unitarily equivalent to the multiplication operator ${\cal M}_z$ on a reproducing kernel Hilbert space $\cal H$ of vector-valued analytic functions defined on a disc $D,$ with center origin and radius $r(L)^{-1},$ where $L=T^{\prime^*}$ and $r(L)$ is the spectral radius of $L.$ Let $L_t=S_t^{\prime^*}$ for every $t>0.$ 
We now give formula for the spectral radius of $S_t,~r(S_t)$ in terms of the symbol $\varphi.$
Recall that $\|S_t\|=\|\varphi_t\|_\infty.$
Observe that $$ \|S_t^n\|=\|S_{nt}\|=\Big|\Big|\sqrt {\frac{\varphi(x)}{\varphi(x-nt)}}\Big|\Big|_\infty.$$
Hence by the spectral radius formula, $$r(S_t)= \lim_{n\rightarrow \infty}\|S_t^n\|^{\frac{1}{n}}=\lim_{n\rightarrow \infty}\Big|\Big|\sqrt {\frac{\varphi(x)}{\varphi(x-nt)}}\Big|\Big|_\infty^{\frac{1}{n}}.$$ 
Now $$r(L_t)=r(S_t^{\prime^*})=r(S_t^\prime)=\lim_{n\rightarrow \infty}\Big|\Big|\sqrt {\frac{\varphi(x-nt)}{\varphi(x)}}\Big|\Big|_\infty^{\frac{1}{n}}.$$ 

Recall that $E=$ ker $S_t^*.$ By injectivity of $S_t,$ $S_t^nE \neq \{0\}$ for all positive integers $n.$ Note that for a left invertible operator $S_t,$ the subspace $S_t^nE$ is a closed subspace of $L^2(\mathbb R_+),$ for every positive integer $n.$ We now prove a lemma. 
\begin{lemma} \label{p28} The subspaces $\{S_t^nE\}_{n=0}^\infty$ are mutually orthogonal. \end{lemma}
\begin{proof} We first prove that $S_t^nE\subseteq \chi_{[nt,(n+1)t)}L^2(\mathbb R_+)$ for all $n\geq 0.$
Let $f\in E.$ By Lemma \ref{p27}, $f=\chi_{[0,t)}g$ almost everywhere on $\mathbb R_+$ for some function $g\in L^2(\mathbb R_+).$ Then for any $n\geq 0,$
\begin{equation*}
(S_t^nf)(x)=(S_t^n\chi_{[0,t)}g)(x) =
\begin{cases}
\displaystyle \sqrt{\frac{\varphi(x)}{\varphi(x-nt)}}\chi_{[0,t)}(x-nt)g(x-nt) & \text{if~ $x\geq nt$},\\
0 & \text{if~ $x<nt$}.
\end{cases}
\end{equation*}
Therefore
\begin{equation*}
(S_t^nf)(x)=
\begin{cases}
\displaystyle \sqrt{\frac{\varphi(x)}{\varphi(x-nt)}}g(x-nt) & \text{if~ $x\in [nt,(n+1)t)$},\\
0 & \text{otherwise}.
\end{cases}
\end{equation*}
Therefore $S_t^nf \in \chi_{[nt,(n+1)t)}L^2(\mathbb R_+).$
Hence $S_t^nE\subseteq \chi_{[nt,(n+1)t)}L^2(\mathbb R_+)$ for all $n\geq 0.$
Note that for $m\neq n,$ the intervals $[mt,(m+1)t)$ and $[nt,(n+1)t)$ are disjoint. Therefore the functions $\chi_{[mt,(m+1)t)}$ and $\chi_{[nt,(n+1)t)}$ are orthogonal. Hence for $m\neq n,~(S_t^mE) \bot (S_t^nE).$ 
%Hence the subspaces $\{S_t^nE\}$ are mutually orthogonal. 
\end{proof}

\begin{remark} \label{p29} For a left invertible semigroup $\{S_t\},$ the application of Lemma \ref{p28}  to a semigroup $\{S_t^\prime\}$ yields the subspaces $\{{S_t^\prime}^nE\}_{n=0}^\infty$ are also mutually orthogonal. \end{remark}

The following theorem describes the analytic model for the left invertible weighted translation semigroup $\{S_t\}.$ The reader may compare this theorem with \cite[Theorem 2.7]{CT} for a similar model developed in the context of weighted shifts on directed trees.  
\begin{theorem} \label{p21} Let $\{S_t\}$ be a weighted translation semigroup with symbol $\varphi.$ Assume that for any given $t\in \mathbb R_+,~\inf_x \frac{\varphi(x+t)}{\varphi(x)} > 0.$ Let $S_t^\prime$ be a Cauchy dual of the operator $S_t$ and $L_t={S_t^\prime}^*.$
Let $E=$ ker $S_t^*.$ Then there exist a reproducing kernel Hilbert space $\cal H$ of E-valued analytic functions defined on a disc $D_r,$ a disc with center origin and radius $r(L_t)^{-1}$ and a unitary operator $U:L^2(\mathbb R_+)\rightarrow \cal H$ such that ${\cal M}_zU=US_t,$ where the operator ${\cal M}_z$ denotes the multiplication by $z$ on $\cal H.$ 
Further, $U$ maps $E$ onto the subspace $\cal E$ of E-valued constant functions in $\cal H$ satisfying $(Ue)(z)=e$ for all $z\in D_r$ and for every $e\in E.$ We have the following:\\
(i) The reproducing kernel $\displaystyle k_{\cal H}:D_r\times D_r \rightarrow B(E)$ associated with $\cal H$ satisfies for any $e\in E$ and $\lambda \in D_r,~~k_{\cal H}(.,\lambda)e\in \cal H$ and for any $f\in L^2(\mathbb R_+),$ $$\left\langle (Uf)(\lambda), e \right\rangle_E=\left\langle Uf, k_{\cal H}(.,\lambda)e\right\rangle_{\cal H}.$$
(ii) The reproducing kernel $\displaystyle k_{\cal H}(z,\lambda)$ is the diagonal operator on $E$ in the sense: 
$$(k_{\cal H}(z,\lambda)e)(x)=\left(\sum_{n=0}^\infty \frac{\varphi(x)}{\varphi(x+nt)}z^n\overline{\lambda}^n\right)e(x),~~e\in E,~x\in \mathbb R_+.$$
(iii) The E-valued polynomials in $z$ are dense in $\cal H.$ \\
%$${\cal H}=\bigvee_{n\geq 0} \{z^ne~:~e\in {\cal E}\}.$$
(iv) The Hilbert space $\cal H$ admits an orthonormal basis consisting of E-valued polynomials in $z.$ \end{theorem}
\begin{proof}
The application of Shimorin's construction \cite{Sh} to the left invertible analytic operator $S_t$ gives the proof of part (i).

A proof of (ii) involves a computation of the reproducing kernel for the RKHS $\cal H$ associated to the operator $S_t,t>0.$ 
The following formula for the reproducing kernel is given in \cite[Corollary 2.14]{Sh}.
For $e\in E,~z,\lambda \in D,$
$$k_{\cal H}(z,\lambda)e =\sum_{n,k\geq 0}(PL^n(L^*)^ki_E)ez^n\overline{\lambda}^k$$
where $i_E$ is the embedding of $E$ into $L^2(\mathbb R_+).$ 
Observe that $$ (L_tf)(x)=\sqrt{\frac{\varphi(x)}{\varphi(x+t)}}f(x+t) ~~\mbox{for all}~ x\geq 0$$ and
\begin{equation*}
L_t^*f(x) =
\begin{cases}
\displaystyle \sqrt {\frac{\varphi(x-t)}{\varphi(x)}}f(x-t) & \text{if~ $x\geq t$}\\
0 & \text{if~ $x<t$}.
\end{cases}
\end{equation*}
Note that $$ L_tL_t^*f(x) =\frac{\varphi(x)}{\varphi(x+t)}f(x)~~\mbox{for all}~ x\geq 0.$$
It is easy to see that for any positive integer $k,$ $$\displaystyle L^k(L^*)^kf(x) =\frac{\varphi(x)}{\varphi(x+kt)}f(x)~~\mbox{for all}~ x\geq 0.$$
Using Remark \ref{p29}, we prove that $(PL_t^j{L_t^*}^k)|_E=0$ for all non-negative integers $j\neq k.$
For any $f,g\in E$ and non-negative integers $j\neq k,$
\beqn
\left\langle PL_t^j{L_t^*}^kf,g\right\rangle_E & = & \left\langle P{{S_t^\prime}^*}^j{S_t^\prime}^kf,g\right\rangle_E = \left\langle {{S_t^\prime}^*}^j{S_t^\prime}^kf,g\right\rangle_{L^2(\mathbb R_+)} \\
& = & \left\langle {S_t^\prime}^kf, {S_t^\prime}^jg\right\rangle_{L^2(\mathbb R_+)}=0. 
\eeqn 
Therefore, $(PL_t^j{L_t^*}^k)|_E=0$ for all non-negative integers $j\neq k.$
 
We now compute the reproducing kernel for the RKHS $\cal H$ associated to the operator $S_t,t>0.$
Note that $E=$ker $S_t^*=$ker $L_t$ is infinite dimensional. For $e\in E,~z,\lambda \in D_r$ and $x\in \mathbb R_+,$
\beqn
 k_{\cal H}(z,\lambda)e(x) &=& \sum_{n,k\geq 0}(PL_t^n(L_t^*)^ki_E)e(x)z^n\overline{\lambda}^k
 = \sum_{n,k\geq 0}(PL_t^n(L_t^*)^k)e(x)z^n\overline{\lambda}^k \\
 &=& \sum_{n=0}^\infty (PL_t^n(L_t^*)^n)e(x)z^n\overline{\lambda}^n
 = \sum_{n=0}^\infty \frac{\varphi(x)}{\varphi(x+nt)}Pe(x)z^n\overline{\lambda}^n \\ 
 &=& \sum_{n=0}^\infty \frac{\varphi(x)}{\varphi(x+nt)}e(x)z^n\overline{\lambda}^n
 = e(x)\sum_{n=0}^\infty \frac{\varphi(x)}{\varphi(x+nt)}z^n\overline{\lambda}^n. \eeqn
Hence $$ k_{\cal H}(z,\lambda)e(x)=\left(\sum_{n=0}^\infty \frac{\varphi(x)}{\varphi(x+nt)}z^n\overline{\lambda}^n\right)e(x).$$

To prove (iii), recall that the operator $S_t,t>0$ possesses the wandering subspace property [Proposition \ref{p19}]. Therefore 
$$L^2(\mathbb R_+)=\bigvee_{n\geq 0} S_t^n(E).$$
Since ${\cal M}_z$ is unitarily equivalent to $S_t,$ it follows that
$${\cal H}=\bigvee_{n\geq 0} {\cal M}_z^n({\cal E}).$$
This proves that the E-valued polynomials in $z$ are dense in $\cal H.$

For proving (iv), we need to describe the orthonormal basis of $L^2(\mathbb R_+)$.
We begin with the following function on $\mathbb R.$ Let
\begin{equation*}
\psi (x) =
\begin{cases}
\displaystyle 1 & \text{if~ $0 \leq x < \frac{1}{2}$}, \\
\displaystyle -1 & \text{if~ $\frac{1}{2}\leq x < 1$}, \\
0 & \text{otherwise }.
\end{cases}
\end{equation*}
The function $\psi$ is referred as basic Haar function in the literature \cite{Pi}. Then the family of Haar functions given by $\{\psi_{jk} (x)= 2^{j/2} \psi (2^j x-k) \},$ where $j,k$ are integers forms an orthonormal basis of $L^2(\mathbb R)$ \cite[Theorem 6.3.4]{Pi}.
Define $\widetilde{\psi_{jk}} :\mathbb R_+\rightarrow \mathbb R$ by $\widetilde{\psi_{jk}}(x)=\psi_{jk} (x)$ for an integer $j$ and a non-negative integer $k.$   
It can be verified that the set ${\cal B}=\{\widetilde{\psi_{jk}}\},$ where $j$ is an integer and $k$ is a non-negative integer is an orthonormal basis of $L^2({\mathbb R_+}).$
A unitary operator $U:L^2(\mathbb R_+)\rightarrow \cal H$ is defined as follows: for each $f\in L^2(\mathbb R_+),$ define an $E-$valued function $U_f$ on the disc $D_r$ as follows:
$$(U_f)(z)=\sum_{n=0}^\infty (PL_t^n f)z^n$$
where $P$ is an orthogonal projection on $E$ and $U(f)=U_f.$ 
Since $U$ is a unitary operator and $\cal B$ is an orthonormal basis of $L^2(\mathbb R_+),$ the set $\{U\widetilde{\psi_{jk}}\},$ where $j$ is an integer and $k$ is a non-negative integer is an orthonormal basis of $\cal H.$
Note that $$(U\widetilde{\psi_{jk}})(z)=\sum_{n=0}^\infty (PL_t^n \widetilde{\psi_{jk}})z^n$$ for some integer $j$ and some non-negative integer $k.$
We now prove that the expression $(U\widetilde{\psi_{jk}})(z)$ contains only finitely many non-zero terms. 
Note that
$$L_t^n\widetilde{\psi_{jk}}(x) = \sqrt {\frac{\varphi(x)}{\varphi(x+nt)}}\widetilde{\psi_{jk}} (x+nt)~\mbox{for all}~ x\geq 0.$$
Therefore $L_t^n\widetilde{\psi_{jk}}=0 ~~{\rm if}~~ nt > \frac{k+1}{2^j}.$
Hence $(U\widetilde{\psi_{jk}})(z)$ is a polynomial of degree at most $[\frac{k+1}{2^j}].$ Thus the Hilbert space $\cal H$ admits an orthonormal basis consisting of E-valued polynomials in $z.$ This proves (iv).
\end{proof}

The following diagram describes the unitary equivalence of a left invertible operator $S_t$ on $L^2(\mathbb R_+)$ with an operator $M_z$ on a reproducing kernel Hilbert space $\cal H$ as stated in Theorem \ref{p21}: 
\begin{center}
\begin{tabular}{ccc}  
$\cal H$ & $\stackrel{M_z}{\longrightarrow}$ & $\cal H$ \\
 & &  \\
$U{\Big\uparrow}$ &           &    ${\Big\downarrow}U^{-1}$  \\
 & &  \\ 
$L^2(\mathbb R_+)$ & $\stackrel{S_t}{\longrightarrow}$ & $L^2(\mathbb R_+)$ 
\end{tabular}
\end{center}
 
We now compute the reproducing kernels in some particular types of weighted translation semigroup $\{S_t\}.$ Recall that the reproducing kernel is a $B(E)$-valued function defined on $D_r\times D_r,$ where $D_r$ is a disc with center origin and radius $r(L_t)^{-1};$ and the formula for $r(L_t)$ is given by
$$r(L_t)=\lim_{n\rightarrow \infty}\Big|\Big|\sqrt {\frac{\varphi(x-nt)}{\varphi(x)}}\Big|\Big|_\infty^{\frac{1}{n}}.$$ 
\begin{example} Let $\{S_t\}$ be a weighted translation semigroup with symbol $\varphi.$
\begin{enumerate} 
\item Let $\varphi(x)=c,$ where $c$ is constant. 
%Then weight function $\varphi_t(x)=1.$ 
%In this special case, the semigroup $\{S_t\}$ is denoted by $\{U_t\}.$ 
In this case, $\frac{\varphi(x)}{\varphi(x+nt)}=1.$ 
Then $r(L_t)=1$ and $D_r$ is the open unit disc. 
%We observe that the reproducing kernel $k_{\cal H}(z,\lambda)$ is the $B(E)$-valued Cauchy kernel and the reproducing kernel Hilbert space $\cal H$ is the E-valued Hardy space of the open unit disc.
It is easy to see that for $e\in E$ and $z,\lambda \in D_r,$
%We now compute the reproducing kernel for the RKHS associated to the semigroup $\{U_t\}.$ For $e\in E,$
%\beqn
%k_{\cal H}(z,\lambda)e(x)=\sum_{n,k\geq 0}(PL^n(L^*)^ki_E)e(x)z^n\overline{\lambda}^k
%=\sum_{n,k\geq 0}(PL^n(L^*)^k)e(x)z^n\overline{\lambda}^k \\
%=\sum_{n=0}^\infty (PL^n(L^*)^n)e(x)z^n\overline{\lambda}^n
%=\sum_{n=0}^\infty Pe(x)z^n\overline{\lambda}^n 
%=\sum_{n=0}^\infty e(x)z^n\overline{\lambda}^n
%=e(x)\sum_{n=0}^\infty z^n\overline{\lambda}^n. 
%\eeqn
$$ k_{\cal H}(z,\lambda)e(x)=\left(\sum_{n=0}^\infty z^n\overline{\lambda}^n\right)e(x)=\frac{1}{1-z\overline{\lambda}}e(x).$$
%Thus, in this case, $\displaystyle k_{\cal H}(z,\lambda)$ is the $B(E)$ valued Cauchy kernel and $\cal H$ is the E-valued Hardy space of the open unit disc.
\item Let $\varphi(x)=x+1.$ 
In this case, $$\displaystyle \frac{\varphi(x)}{\varphi(x+nt)}=\frac{x+1}{x+1+nt}=1-\frac{nt}{x+1+nt}.$$
Therefore $r(L_t)=1$ and $D_r$ is the open unit disc.
For $e\in E$ and $z,\lambda \in D_r,$ the computations reveal that 
\beqn
k_{\cal H}(z,\lambda)e(x)
%&=& \left(\sum_{n=0}^\infty \frac{\varphi(x)}{\varphi(x+nt)}z^n\overline{\lambda}^n\right)e(x) \\
%&=& \left(\sum_{n=0}^\infty (1-\frac{nt}{x+1+nt})z^n\overline{\lambda}^n\right)e(x)\\
&=& \left(\frac{1}{1-z\overline{\lambda}}-\sum_{n=0}^\infty \frac{nt}{x+1+nt}z^n\overline{\lambda}^n\right)e(x).
\eeqn

\item Let $\displaystyle \varphi(x)=\frac{1}{x+1}.$ 
In this case, $$\frac{\varphi(x)}{\varphi(x+nt)}=\frac{x+1+nt}{x+1}=1+\frac{nt}{x+1}.$$
%Hence
Therefore $r(L_t)=1$ and $D_r$ is the open unit disc.
For $e\in E$ and $z,\lambda \in D_r,$ it can be seen that  
\beqn
k_{\cal H}(z,\lambda)e(x)
%&=& \left(\sum_{n=0}^\infty \frac{\varphi(x)}{\varphi(x+nt)}z^n\overline{\lambda}^n\right)e(x) \\
%&=& \left(\sum_{n=0}^\infty (1+\frac{nt}{x+1})z^n\overline{\lambda}^n\right)e(x) \\
%&=& \left(\sum_{n=0}^\infty z^n\overline{\lambda}^n +\frac{t}{x+1}z\overline{\lambda}\sum_{n=0}^\infty nz^{n-1}\overline{\lambda}^{n-}\right)e(x) \\
&=& \left(\frac{1}{1-z\overline{\lambda}}+\frac{t}{x+1}\frac{z\overline{\lambda}}{(1-z\overline{\lambda})^2}\right)e(x).
\eeqn

\item Let
\begin{equation*}
\varphi(x) =
\begin{cases}
x+1 & \text{if~ $0\leq x\leq 1$},\\
2 & \text{if~ $x>1$}.
\end{cases}
\end{equation*}
If $x\geq 1,$ then $x+nt\geq 1$ for all $n \in \mathbb N,t \in \mathbb R_+.$ In this case, $\frac{\varphi(x)}{\varphi(x+nt)}=1.$ For $x<1,$ there exists $N\in \mathbb N$ such that $x+Nt\leq 1,$ but $x+(N+1)t>1.$ Then 
\begin{equation*}
\frac{\varphi(x)}{\varphi(x+nt)}=
\begin{cases}
\frac{x+1}{x+nt+1} & \text{if~ $n\leq N$},\\
\frac{x+1}{2} & \text{if~ $n>N$}.
\end{cases}
\end{equation*}
Observe that $r(L_t)=1$ and $D_r$ is the open unit disc.
It can be seen that for $e\in E$ and $z,\lambda \in D_r,x\geq 1,$
$$ k_{\cal H}(z,\lambda)e(x)=\left(\sum_{n=0}^\infty z^n\overline{\lambda}^n\right)e(x)=\frac{1}{1-z\overline{\lambda}}e(x)$$
and for $x<1,$
\beqn
k_{\cal H}(z,\lambda)e(x) &=& \left(\sum_{n=0}^N \frac{x+1}{x+nt+1} z^n\overline{\lambda}^n + \sum_{n=N+1}^\infty \frac{x+1}{2} z^n\overline{\lambda}^n \right)e(x) \\
&=& \left(\sum_{n=0}^N \left(\frac{x+1}{x+nt+1}-\frac{x+1}{2} \right)z^n\overline{\lambda}^n + \frac{x+1}{2} \frac{1}{1-z\overline{\lambda}} \right) e(x).
\eeqn

\item Let $\varphi (x)=a^x,~a>1.$
In this case, $\frac{\varphi(x)}{\varphi(x+nt)}=\frac{a^x}{a^{x+nt}}=a^{-nt}.$
Therefore $r(L_t)=a^{-t}$ and $D_r$ is the disc with radius $a^t.$
For $e\in E$ and $z,\lambda \in D_r,$ we observe that  
\beqn
k_{\cal H}(z,\lambda)e(x)
%=\left(\sum_{n=0}^\infty \frac{\varphi(x)}{\varphi(x+nt)}z^n\overline{\lambda}^n\right)e(x)
=\left(\sum_{n=0}^\infty a^{-nt}z^n\overline{\lambda}^n\right)e(x)
=\frac{1}{1-a^{-t}z\overline{\lambda}}e(x).
\eeqn
\end{enumerate} 
In the light of \cite[Corollary 3.3]{PS}, we observe that in example (1) the semigroup is an isometry, in example (2) the semigroup is a 2-isometry , in example (3) the semigroup is a subnormal contraction, in example (4) the semigroup is 2-hyperexpansive and in example (5), the semigroup is alternatingly hyperexpansive. \eop
\end{example}

\subsection{Operator Valued Weighted Shift}
An operator valued weighted shift is a generalization of a weighted shift operator in the sense that the weight sequence is a sequence of operators. In this subsection, we prove that for every $t>0,$ a  left invertible operator $S_t$ is an operator valued weighted shift. We begin with the definition of an operator valued weighted shift. Let $H$ be a nonzero Hilbert space. The Hilbert space denoted by $l^2_H$ is the Hilbert space of all vector sequences $\{h_n\}_{n=0}^\infty \subseteq H$ such that $\sum_{n=0}^\infty ||h_n||^2 < \infty,$ equipped with the standard inner product 
\beqn \left\langle \{g_n\}_{n=0}^\infty, \{h_n\}_{n=0}^\infty \right\rangle 
=\sum_{n=0}^\infty \langle g_n,h_n \rangle ,~~~~~\{g_n\}_{n=0}^\infty, \{h_n\}_{n=0}^\infty \in l^2_H.
\eeqn
If $\{W_n\}_{n=0}^\infty \subseteq {\cal B(H)}$ is a uniformly bounded sequence of operators, then the operator $W\in B(l^2_H)$ defined by $$W(x_0,x_1,\cdots  )=(0,W_0x_0,W_1x_1,\cdots ),~~~~~~(x_0,x_1,\cdots  )\in l^2_H$$ is called an {\it operator valued weighted shift} with weights $\{W_n\}.$ For basic theory of operator valued weighted shifts, the reader is referred to \cite{Jb2},\cite{L}. 

\begin{theorem} \label{p26} For every $t>0,$ a left invertible operator $S_t$ is an operator valued weighted shift.\end{theorem}
\begin{proof} Recall that for $t>0,$ the operator $S_t$ possesses the wandering subspace property [Proposition \ref{p19}]. That is $\displaystyle L^2(\mathbb R_+)=\vee_{n=0}^\infty S_t^nE,$ where the subspace $E=$ ker $S_t^*.$ 
In addition, the closed subspaces $\{S_t^nE\}$ are mutually orthogonal [Lemma \ref{p28}]. Hence $\displaystyle L^2(\mathbb R_+)=\oplus_{n=0}^\infty S_t^nE.$ 
Now by the similar argument as given in \cite[Theorem 3.8]{ACJ}, the operator $S_t,t>0$ is an operator valued weighted shift. 
\end{proof}

\section{Spectral Picture}
In this section, we describe the spectral picture of an operator in a left invertible semigroup $\{S_t\}$ and compute the spectrum in some particular examples. As in case of a weighted shift operator, it turns out that the spectrum of $S_t,t>0$ is a disc and the point spectrum of $S_t,t>0$ is empty. In this section, we assume that $t$ is a positive real number.

We use the following notation in the sequel: 
$\sigma (T):$ Spectrum of $T,$
$\sigma_{ap} (T):$ Approximate point spectrum of $T,$
$\sigma_p (T):$ Point spectrum of $T,$
$\sigma_e (T):$ Essential spectrum of $T.$
 
We begin with an observation about the circular symmetry of the spectrum of $S_t.$
%The following result is useful in proving the circular symmetry of the spectrum.
\begin{proposition} \label{p31} The operator $S_t$ is unitarily equivalent to the operator $e^{-i\theta t}S_t$ for any real number $\theta .$ \end{proposition} 
\begin{proof} For a real number $\theta,$ define the map $M_\theta :L^2(\mathbb R_+)\rightarrow L^2(\mathbb R_+)$ as $(M_\theta f)(x)=e^{i\theta x}f(x).$ Then clearly $M_\theta$ is a unitary operator.
We claim that $\displaystyle M^*_\theta S_tM_\theta=e^{-i\theta t}S_t.$
Note that
\beqn
((M^*_\theta S_tM_\theta)f)(x)&=& M^*_\theta (S_tM_\theta f)(x)=e^{-i\theta x}(S_tM_\theta f)(x) \\ 
& =&
\begin{cases}
\displaystyle e^{-i\theta x}\sqrt {\frac{\varphi(x)}{\varphi(x-t)}}(M_\theta f)(x-t) & \text{if~ $x\geq t,$}\\
0 & \text{if~ $x<t$}
\end{cases}
%\end{equation*}
%\begin{equation*}
\\ &=&
\begin{cases}
\displaystyle e^{-i\theta x}\sqrt {\frac{\varphi(x)}{\varphi(x-t)}}e^{i\theta (x-t)}f(x-t) & \text{if~ $x\geq t,$}\\
0 & \text{if~ $x<t$}
\end{cases}
%\end{equation*}
\\ &=& ((e^{-i\theta t}S_t)f)(x).\eeqn 
Hence the operator $S_t$ is unitarily equivalent to the operator $e^{-i\theta t}S_t.$
\end{proof} 

\begin{remark} From Proposition \ref{p31}, the spectrum of $S_t$ as well as various spectral parts have circular symmetry about the origin. Also note that the operator $S_t$ is not onto. Indeed, the characteristic function of the interval $[0,t)$ does not belong to the range of $S_t.$ Therefore $0\in \sigma (S_t).$\end{remark}

We now describe the spectrum of a left invertible operator $S_t.$
\begin{proposition} \label{p23} Let $S_t$ be a left invertible operator.
\begin{enumerate} 
\item The point spectrum of the operator $S_t,\sigma_p(S_t)$ is empty.
\item The point spectrum of $S_t^*$ contains the disc $D_r$ with center origin and radius $r(L_t)^{-1}.$
\item The spectrum of the operator $S_t,~\sigma (S_t)$ is a closed disc with center origin and radius $r(S_t).$
\end{enumerate} \end{proposition} 
\begin{proof} \
\begin{enumerate} 
\item The proof readily follows from the fact that for any complex number $\lambda,$ the operator $S_t-\lambda I$ is injective.  
\item The fact that the operator $S_t$ is unitarily equivalent to the operator ${\cal M}_z$ on the reproducing kernel Hilbert space $\cal H$ is used in the following proof. By Theorem \ref{p21} (i), for $f\in L^2(\mathbb R_+),~e\in E$ and $w\in D_r,$
$$ \left\langle U_f, M_z^*k_{\cal H}(.,w)e\right\rangle_{\cal H} = \left\langle M_z U_f, k_{\cal H}(.,w)e\right\rangle_{\cal H}
=\left\langle wU_f(w), e\right\rangle_E
= \left\langle U_f, \overline{w}k_{\cal H}(.,w)e\right\rangle_{\cal H} .$$
Thus $M_z^*k_{\cal H}(.,w)e=\overline{w}k_{\cal H}(.,w)e$ for all $w\in D_r$ and $e\in E.$ Hence the point spectrum of $S_t^*$ contains the disc $D_r.$
\item By part (2), $D_r\subseteq \sigma_p(S_t^*).$ The circular symmetry of spectrum implies that $\sigma(S_t)=\sigma(S_t^*).$ Therefore $D_r\subseteq \sigma(S_t).$ By \cite[Lemma 5.3]{CT}, the spectrum of an analytic operator is always connected. Hence, $\sigma (S_t)$ is a closed disc with radius $r(S_t).$ \end{enumerate} 
This completes the proof of the theorem.
\end{proof}
At this stage, we point out that Proposition \ref{p23} (1) and (3), proved independently here, can be also looked upon as a consequence of \cite[Lemma 2.2(iii)]{L}. 
We now turn our attention towards the approximate point spectrum of an operator $S_t.$ Recall that for any operator $T,$ $$m(T):=\inf \{||Tf||:||f||=1\} ~\rm{and} ~ r_1(T)=\lim_{n\rightarrow \infty}[m(T^n)]^{\frac{1}{n}}.$$
At this stage, we record the following well known result useful in determining the approximate point spectrum of an operator \cite{Ri}. 
For any operator $T,$ the approximate point spectrum $$\sigma_{ap} (T)\subseteq \{z~:~r_1(T)\leq |z| \leq r(T)\}.$$
We now compute the spectra of the semigroups $\{S_t\}$ associated to some special types of symbols $\varphi .$
\begin{example} Let $\{S_t\}$ be a weighted translation semigroup with symbol $\varphi.$  
\begin{enumerate}
\item Let $\varphi(x)=c.$ 
It is easy to see that $r_1(S_t)=r(S_t)=1.$ Therefore $\sigma_{ap} (S_t)$ is a unit circle and $\sigma (S_t)$ is a closed unit disc.
\item Let $\varphi(x)=x+1.$
In this case, $r_1(S_t)=r(S_t)=1.$ Hence $\sigma_{ap} (S_t)$ is a unit circle and $\sigma (S_t)$ is a closed unit disc. 
\item Let $\varphi(x)=e^{2x}.$ 
In this case, $r_1(S_t)=r(S_t)=e^t.$ Thus $\sigma_{ap} (S_t)$ is a circle with radius $e^t$ and $\sigma (S_t)$ is a closed disc with radius $e^t.$ 
\end{enumerate} 
Note that in example (1) the semigroup is an isometry, in example (2) the semigroup is a 2-isometry and in example (3), the semigroup is alternatingly hyperexpansive (refer \cite[Corollary 3.3]{PS}).
\eop
\end{example} 

\begin{remark} Atkinson's theorem \cite[Theorem 5.17]{D} asserts that an operator $T\in \cal B(H)$ is essentially invertible if and only if ker $T$ is finite dimensional, ker $T^*$ is finite dimensional and range of $T$ is closed. 
Here, we know that  ker $S_t=\{0\}$ and ker ${S_t}^*$ is infinite dimensional.
As a consequence, the operator $S_t$ is not essentially invertible and thus $0\in  ~\sigma_e(S_t).$  We know that for a weighted shift operator $T,$ ker $T=\{0\}$ and ker $T^*$ is one dimensional. Therefore if the range of a weighted shift operator $T$ is closed, then it is essentially invertible. If $T$ is a completely hyperexpansive weighted shift, then the essential spectrum of $T,~\sigma_e(T)$ is a unit circle \cite[Proposition 5]{At}. If $T$ is a hyponormal weighted shift, then $\sigma_e(T)$ is a circle with center origin and radius $||T||$ \cite[Theorem 6.7(a)]{Co}. Thus in these cases, the essesntial spectrum of a weighted shift operator and that of a weighted translation semigroup are not the same.
\end{remark} 

%\section{Concluding Remarks} 
%The present work attempts to study the operator theoretic properties  of an operator of the form $S_t$ with a motivation to compare its behaviour with weighted shifts. The appearance of a function in place of a sequence allows one to use different techniques, not available in the discrete case. The authors believe that these techniques might be employed to tackle some problems about weighted shifts. In particular, the problem of characterizing $m$-isomeric transformations whose Cauchy dual is a contractive subnormal operator is challenging even in the case of a weighted shift operator. Though some special cases of this problem have been discussed in \cite{AC}, the general case still remains unanswered. Further, the operators of the type $S_t$ have been shown to be rather different from weighted shifts and the study of such operators gives a variety of new examples of operators. The authors propose to undertake a deeper study of this class. In particular, we shall attempt to extend the notion to a multi-variable set-up.    

\subsection*{Acknowledgment} The authors are thankful to an unknown referee for pointing out some corrections in the original version and offering some useful suggestions. The authors would also like to thank Sameer Chavan for several useful discussions throughout the preparation of this paper.

\end{document}